\documentclass[letterpaper]{article}
\usepackage{amsmath}
\usepackage{amsthm}
\usepackage{amsfonts}
\usepackage{amssymb}

\newtheorem{theorem}{Theorem}

\newtheorem{definition}[theorem]{Definition}
\newtheorem{lemma}[theorem]{Lemma}

\theoremstyle{remark}
\newtheorem{remark}[theorem]{Remark}
\renewenvironment{proof}[1][Proof]{\noindent\textbf{#1.} }{\ \rule{0.5em}{0.5em}}

\numberwithin{equation}{section}
\numberwithin{theorem}{section}

\begin{document}

\title{Adapted complex structures and the geodesic flow}
\author{Brian C. Hall\thanks{\emph{University of Notre Dame, 255 Hurley Building, Notre Dame, IN \ 46556 \ USA}\newline E-mail: bhall@nd.edu}\ \ and William D. Kirwin\thanks{\emph{Center for Mathematical Analysis, Geometry and Dynamical Systems, Insituto Superior T\'ecnico, Av. Rovisco Pais, 1049-001 Lisbon, Portugal}.\newline E-mail: will.kirwin@gmail.com\newline The first author was supported in part by NSF grant DMS-0555862. The second author would like to thank the University of Hong Kong and the Max Planck Institute for Mathematics in the Sciences for their hospitality during the preparation of this paper.}}
\date{}

\maketitle

\begin{abstract}
In this paper, we give a new construction of the \textit{adapted complex structure} on a neighborhood of the zero section in the tangent bundle of a compact, real-analytic Riemannian manifold. Motivated by the ``complexifier'' approach of T. Thiemann as well as certain formulas of V. Guillemin and M. Stenzel, we obtain the polarization associated to the adapted complex structure by applying the ``imaginary-time geodesic flow'' to the vertical polarization. Meanwhile, at the level of functions, we show that every holomorphic function is obtained from a function that is constant along the fibers by ``composition with the imaginary-time geodesic flow.'' We give several equivalent interpretations of this composition, including a convergent power series in the vector field generating the geodesic flow.
\end{abstract}

{\small
}

\section{Introduction}

Let $(M^{n},g)$ be a compact, connected, real-analytic Riemannian manifold. Let $TM$ be the tangent bundle of $M$, permanently identified with the cotangent bundle by means of the metric. We denote points in $TM$ as $z=(x,v)$, where $x\in M$ and $v\in T_{x}M.$ Let $T^{R}M$ denote the set of points $(x,v)\in TM$ with $g(v,v)<R^{2},$ i.e., the set of tangent vectors of length less than $R.$ Let $E$ be the ``energy'' function defined by $E(x,v)=\frac{1}{2}g(v,v).$ Let $\Phi_{\sigma}:TM\rightarrow TM$ then denote the time-$\sigma$ geodesic flow, which is the Hamiltonian flow associated to the function $E.$

Working independently, L. Lempert and R. Sz\H{o}ke \cite{Lempert-Szoke,Szoke} and V. Guillemin and M. Stenzel \cite{Guillemin-Stenzel1,Guillemin-Stenzel2} introduced natural complex structures defined on $T^{R}M,$ for some sufficiently small $R.$ Lempert--Sz\H{o}ke characterized their complex structure in terms of geodesics---see Theorem \ref{adapted.thm}. Meanwhile, Guillemin--Stenzel characterized their complex structure in terms of a K\"{a}hler potential and an involution---see Section \ref{kahler.sec}. Both sets of authors show that there is a unique such complex structure on $T^{R}M$ for some sufficiently small $R.$ It turns out that the two characterizations are equivalent, and complex structures satisfying these characterizations are referred to as \textit{adapted complex structures}. The adapted complex structure fits together with the canonical symplectic structure on $TM\cong T^{\ast}M$ to make $TM$ into a K\"{a}hler manifold. Numerous authors have studied adapted complex structures from various points of view; see, for example, \cite{Aguilar01,Burns-Hind01,Szoke01,Halverscheid01,Totaro03}.

There are certain special cases in which the adapted complex structure exists on all of $TM,$ including the case in which $M$ is a compact Lie group with a bi-invariant metric. In the compact group case, one obtains a nice quantization of $TM$ by performing geometric quantization using the K\"{a}hler structure associated to the adapted complex structure. The resulting Hilbert space can be identified \cite{Hall02} with the generalized Segal--Bargmann space introduced in \cite{Hall94,Hall97}. See also \cite{Florentino-Matias-Mourao-Nunes05,Florentino-Matias-Mourao-Nunes06}, \cite{Huebschmann}, and the discussion in Section \ref{quant.sec}.

In this paper, we build on the ``complexifier'' method of T. Thiemann \cite{Thiemann96, Thiemann1,Thiemann06} to give a new construction of the adapted complex structure. Our main construction, though, is phrased in terms of involutive Lagrangian distributions.\footnote{In the language of geometric quantization, an involutive Lagrangian distribution is called a \textit{polarization.}} If there is a K\"{a}hler structure on a symplectic manifold $N$, then the distribution $P=T^{(1,0)}N$, which determines the complex structure, is an involutive Lagrangian distribution. In the case $N=TM$, there is one obvious involutive Lagrangian distribution, called the \textit{vertical distribution}, in which the distribution $P_{z}$ at the point $z\in TM$ is the (complexification of) the vertical subspace.

Our main result is that the involutive Lagrangian distribution associated to the adapted complex structure can be obtained from the vertical distribution by the ``time-$i$ geodesic flow.'' The way we make sense of this is to first push forward the vertical distribution by $\Phi_{\sigma},$ for real values of $\sigma,$ obtaining a family $P_{z}(\sigma)$ of involutive Lagrangian distributions. Then we show that for all $z$ in some $T^{R}M,$ one can analytically continue the map $\sigma\rightarrow P_{z}(\sigma)$ to a holomorphic map of a disk of radius greater than $1$ into the Grassmannian of Lagrangian subspaces of $T_{z}^{\mathbb{C}}TM.$ This analytic continuation serves to define $P_{z}(i)$ on $T^{R}M.$ We then verify that on some possibly smaller tube $T^{R^{\prime}}M,$ there is a complex structure for which $P_{z}(i)$ is, at each point, the $(1,0)$ subspace. Finally, we show that this complex structure satisfies both the Lempert--Sz\H{o}ke conditions and the Guillemin--Stenzel conditions.

Once the construction of the adapted complex structure is complete, we relate the adapted complex structure to the imaginary-time geodesic flow at the level of functions. Although these results have already appeared in the works of Guillemin--Stenzel and Sz\H{o}ke, we include them here for completeness. We show that holomorphic functions with respect to the adapted complex structure can be obtained from functions that are constant on the fibers by ``composing with the time-$i$ geodesic flow.'' What this means is the following. If a function $\psi$ is constant along the leaves of the vertical distribution (the fibers of the tangent bundle), then we consider $\psi\circ\Phi_{\sigma}.$ If $\psi$ is real analytic, then for all $z$ in some $T^{R}M,$ the map $\sigma\rightarrow\psi(\Phi_{\sigma}(z))$ admits an analytic continuation to a disk of radius greater than $1.$ We show that the function $\psi(\Phi_{i}(z)),$ interpreted in terms of the just-mentioned analytic continuation, is holomorphic on $T^{R}M$ and that every holomorphic function on $T^{R}M$ arises in this way. Equivalently, $\psi\circ\Phi_{i}$ can be constructed by a convergent power series in the vector field generating $\Phi_{\sigma}$ (this is essentially Thiemann's approach) or by analytically continuing the exponential map for $M$ (this is contained implicitly in the work of Sz\H{o}ke in the proof of Proposition 3.2 in \cite{Szoke} and in Theorem 3.4 in \cite{Szoke95}, and is also essentially the approach of Guillemin and Stenzel, in Section 5 of \cite{Guillemin-Stenzel2}).

Although the imaginary-time geodesic flow appears already in \cite{Guillemin-Stenzel2} and \cite{Halverscheid01} and in \cite{Thiemann96} at the level of functions, we believe that our approach sheds new light on the subject by giving a self-contained construction of the adapted complex structure using this flow. For example, this way of thinking provides a simple explanation for the formula expressing the adapted complex structure in terms of Jacobi fields (Section \ref{jacobi.sec}), and sheds light on the results of \cite{Florentino-Matias-Mourao-Nunes05,Florentino-Matias-Mourao-Nunes06} (Section \ref{quant.sec}). We also hope that, as suggested by Thiemann, the use of other imaginary-time Hamiltonian flows will lead to the construction of new complex structures.

\section{Main results}
\label{sec:main results}

As mentioned above, we identify the tangent and cotangent bundles $TM\simeq T^*M$ via the metric $g$. Under this identification, we obtain from the canonical $1$-form and symplectic form on $T^*M$ a $1$-form $\Theta$ and symplectic form $\omega=-d\Theta$ on $TM$. Given a function $f\in C^1(M),$ we defined its Hamiltonian vector field $X_f$ by $\omega(X_f,\cdot)=df.$

Consider the energy function $E(x,v)=\tfrac12 g(v,v).$ Denote by $\Phi_\sigma:TM\rightarrow TM$ the time-$\sigma$ flow of its Hamiltonian vector field $X_E.$ As is well known (and can be easily verified by computation in local coordinates, for example), $\Phi_\sigma$ is the geodesic flow; that is, if $\gamma_z:\mathbb{R}\rightarrow M$ denotes the geodesic determined by the point $z\in TM,$ and $\dot\gamma_z$ denotes its derivative, then
\[\Phi_\sigma(z)=\dot\gamma_z(\sigma).\]

Denote by $V_{z}\subset T_{z}^{\mathbb{C}}TM$ the complexification of the vertical tangent space to $TM$ at $z$. (That is, the vertical subspace at $(x,v)\in TM$ is the complexified tangent space to the fiber $T_{x}M.$) Our principal object of interest is the pushforward of the vertical distribution by the geodesic flow. For any real number $\sigma,$ consider
\[
P_{z}(\sigma):=\left(  \Phi_{\sigma}\right)  _{\ast}(V_{\Phi_{-\sigma}(z)})\subset T_{z}^{\mathbb{C}}TM.
\]
For each fixed $z\in TM,$ we are going to analytically continue the map $\sigma\rightarrow P_{z}(\sigma)$ to a holomorphic map of a disk about the origin in $\mathbb{C}$ into the Grassmannian of $n$-dimensional complex subspaces of $T_{z}^{\mathbb{C}}TM.$

The first main result of this paper is the following theorem, a concatenation of Theorems \ref{thm:continuation}, \ref{thm:integrable}, and \ref{thm:pos Lagr}, which are proved in Section \ref{sec:geod flow}.

\begin{theorem}
\label{thm:acs existence}There exist $R>0$ and $\epsilon>0$ such that

\begin{enumerate}
\item For each $z\in T^{R}M,$ the family $\sigma\mapsto P_{z}(\sigma)$ can be analytically continued in $\sigma$ to a disk $D_{1+\epsilon}$ of radius $1+\epsilon$ around the origin, thus giving a meaning to the expression $P_{z}(i).$

\item For each $z\in T^{R}M,$ $P_{z}(i)$ intersects its complex conjugate only at zero.

\item Let $J$ be the unique almost complex structure on $T^{R}M$ such that the restriction of $J$ to $P_{z}(i)$ is multiplication by $i.$ Then $J$ is integrable and fits together with the canonical symplectic form $\omega$ on $TM$ so as to give a K\"{a}hler structure to $T^{R}M.$ This means, in particular, that $P_{z}(i)$ is a Lagrangian subspace relative to $\omega.$
\end{enumerate}
\end{theorem}

The family $P_{z}(i)$ of subspaces constitutes a positive involutive Lagrangian distribution, that is, a \textit{K\"{a}hler distribution.}

Our next results show that the complex structure we construct is the same as the one constructed (independently and from different points of view) by Lempert--Sz\H{o}ke and Guillemin--Stenzel. See also Section \ref{jacobi.sec}, where we show how to compute the complex structure in terms of Jacobi fields, in a way that agrees with the corresponding calculations in \cite{Lempert-Szoke}. We recall first the definition of the adapted complex structure due to Lempert and Sz\H{o}ke.

Consider a point $z=(x,v)$ in $TM$ with $|v|=1.$ Let $S_{R}\subset\mathbb{C}$ denote the strip $\{(\sigma,\tau)|~|\tau|<R\}.$ Define a map $\Psi_{z}:S_{R}\rightarrow T^{R}M$ by $\Psi_{z}(\sigma,\tau)=(\gamma_z(\sigma),\tau\dot{\gamma_z}(\sigma)).$ The following definition is due to Lempert and Sz\H{o}ke \cite[Def. 4.1]{Lempert-Szoke}.

\begin{definition}(Lempert--Sz\H{o}ke)
A complex structure on $T^{R}M$ is called \emph{adapted} (to the metric on $M$) if for all $z=(x,v)$ with $|v|=1,$ $\Psi_{z}$ is holomorphic as a map of $S_{R}\subset\mathbb{C}$ into $T^{R}M.$
\end{definition}

\begin{theorem}
The complex structure on $T^{R}M$ described in Theorem \ref{thm:acs existence} is adapted.
\end{theorem}

In light of the preceding theorem, we will henceforth refer to the complex structure in Theorem \ref{thm:acs existence} as the adapted complex structure on $T^{R}M.$ We should emphasize, though, that our proofs are independent of the work of Lempert--Sz\H{o}ke and of that of Guillemin--Stenzel. The main interest in our results is the clear and unifying geometric picture that the adapted complex structure arises from the time-$i$ complexified geodesic flow.

Given the above result, the next theorem follows from the work of Lempert and Sz\H{o}ke (it is contained in Corollary 5.5, Theorem 5.6 and Corollary 5.7 of \cite{Lempert-Szoke}), although again, we will give an independent proof using our point of view that the adapted complex structure arises from the geodesic flow.

\begin{theorem}
If $R$ is as in Theorem \ref{thm:acs existence}, then the map $\phi(x,v):=(x,-v)$ is antiholomorphic with respect to the complex structure on $T^{R}M$ described in that theorem. Furthermore, the function $\kappa(x,v):=g(v,v)$ is a K\"{a}hler potential for the K\"{a}hler structure described in Point 3 of the theorem, and moreover associated to the canonical $1$-form $\Theta$ on $T^*M,$ that is, $\operatorname{Im}\bar{\partial}\kappa=\Theta.$
\end{theorem}

This result shows that the complex structure agrees with that of Guillemin--Stenzel. (See the theorem on p. 568 of \cite{Guillemin-Stenzel1}.)

Our next main result, from Section \ref{sec:Taylor series}, describes holomorphic functions on $T^{R}M$ in terms of the imaginary-time geodesic flow. Informally, it says that composition of a vertically constant function with the time-$i$ geodesic flow yields a holomorphic function. The first part of the Theorem is essentially a restatement of Proposition 3.2 of \cite{Lempert-Szoke}, whereas the expression (\ref{fc.series}) embodies the ``complexifier'' approach advocated by Thiemann \cite{Thiemann96,Thiemann1}, although Thiemann does not address the question of convergence.

\begin{theorem}
\label{function.thm}Let $R$ be as in Theorem \ref{thm:acs existence} and suppose a function $f:M\rightarrow\mathbb{C}$ admits an entire analytic continuation $f_{\mathbb{C}}$ to $T^{r}M$ for some $r\leq R$. Then for each $z\in T^{r}M$, the map
\[
\sigma\mapsto\left(  f\circ\pi\circ\Phi_{\sigma}\right)  (z)
\]
admits an analytic continuation in $\sigma$ from $\mathbb{R}$ to $\{\sigma+i\tau\in\mathbb{C}:\left\vert \tau\right\vert <r/\left\vert v\right\vert \}.$ Moreover, the function $f_{\mathbb{C}}$ is given by the value of the continuation at $\sigma=i$; this can be expressed informally as
\begin{equation}
f_{\mathbb{C}}(z)=\left(  f\circ\pi\circ\Phi_{i}(z)\right)  . \label{fc.form}
\end{equation}

Furthermore, $f_{\mathbb{C}}(z)$ can be expressed as an absolutely convergent series
\begin{equation}
f_{\mathbb{C}}(z)=\sum_{k=0}^{\infty}\frac{i^{k}}{k!}X_{E}^{k}(f\circ\pi)(z),\label{fc.series}
\end{equation}
where $X_{E}$ is the Hamiltonian vector field associated to $E.$ In this expansion, $(X_{E}^{k}f)(x,v)$ is a homogeneous polynomial of degree $k$ with respect to $v$ for each fixed $x\in M,$ so that (\ref{fc.series}) may be thought of as a ``Taylor series in the fibers.''
\end{theorem}

Recall that we are thinking of the Lagrangian distribution associated to the adapted complex structure as being obtained from the vertical distribution by means of the time-$i$ geodesic flow. Theorem \ref{function.thm} simply expresses this idea at the level of functions. The function $\psi:=f\circ\pi$ is constant along the leaves of the vertical distribution and the function $f_{\mathbb{C}}$ is ``constant along the leaves of the complex distribution'' associated to the adapted complex structure, meaning that the derivative of $f_{\mathbb{C}}$ is zero in the directions of the $(0,1)$ subspace. (It is customary, for example in geometric quantization, to take derivatives in the direction of the \textit{complex conjugate} of the subspaces $P_{z}$ of a Lagrangian distribution.) Theorem \ref{function.thm} says, informally, that $f_{\mathbb{C}}=\psi\circ\Phi_{i}.$

Note that $\pi(\Phi_{\sigma}(x,v))=\exp_{x}(v),$ where $\exp_{x}$ is the geometric exponential map. Thus, (\ref{fc.form}) is formally equivalent to the assertion that $f_{\mathbb{C}}(x,v)=f(\exp_{x}(iv)).$ This expression for $f_{\mathbb{C}}$ is already implicit in the work of Guillemin--Stenzel (Section 5 of \cite{Guillemin-Stenzel2}).

\section{Time-$i$ geodesic flow\label{sec:geod flow}}

In this section, we consider the imaginary-time geodesic flow as it acts on Lagrangian distributions. On the tangent bundle $TM$ (which is permanently identified with the cotangent bundle), we have the vertical distribution, consisting of the vertical subspace of the tangent space to $TM$ at each point. Since the geodesic flow is, for each time, a symplectomorphism, pushing forward the vertical distribution by the time-$\sigma$ geodesic flow gives another Lagrangian distribution. Thus, at each point $z\in TM$, we have a family $P_{z}(\sigma)$ of Lagrangian subspaces of $T_{z}^{\mathbb{C}}TM$ parameterized by the real number $\sigma.$ In this section we will show that for all $z$ in a neighborhood of the zero section, we can analytically continue this family to $\sigma=i.$

Furthermore, we will see that the family of subspaces $P_{z}(i)$ constitutes an involutive complex Lagrangian distribution on $T^{R}M$, and is also the distribution of $(1,0)$ vectors for a complex structure on $TM$ (that is, it intersects its complex conjugate only at $0$). It is not hard to see (Theorem \ref{adapted.thm}) that this complex structure is ``adapted'' in the sense of Lempert--Sz\H{o}ke. It can also be shown (Section \ref{kahler.sec}) that this complex structure satisfies the conditions of Guillemin--Stenzel (the theorem on p. 568 of \cite{Guillemin-Stenzel1}).

\subsection{Basic properties}

For each point $z\in TM,$ let $V_{z}\subset T_{z}^{\mathbb{C}}TM$ denote the \textit{complexification} of the vertical subspace of the tangent space.

\begin{definition}
For each fixed $\sigma\in\mathbb{R}$ and point $z\in TM,$ define the subspace $P_{z}(\sigma)\subset T_{z}^{\mathbb{C}}TM$ by
\begin{equation}
P_{z}(\sigma)=(\Phi_{\sigma})_{\ast}(V_{\Phi_{-\sigma}(z)}). \label{eqn:pushforwardV}
\end{equation}
\end{definition}

We consider the Lagrangian Grassmannian $\mathcal{L}_{z},$ consisting of all $n$-dimensional Lagrangian subspaces of $T_{z}^{\mathbb{C}}TM.$ This is a complex manifold. (It is a complex submanifold of the Grassmannian of all $n$-dimensional complex subspaces of $T_{z}^{\mathbb{C}}TM.$) We then wish to extend the map $\sigma\in\mathbb{R}\rightarrow P_{z}(\sigma)\in\mathcal{L}_{z}$ to a holomorphic map of some domain in $\mathbb{C}$ into $\mathcal{L}_{z}$ and, if possible, evaluate at $\sigma=i.$ To avoid ambiguity in the value of $P_{z}(\sigma)$ at $\sigma=i,$ we limit our attention to points $z$ such that the map $\sigma\rightarrow P_{z}(\sigma)$ can be analytically continued to a disk of radius greater than 1 about the origin.

We now develop the relevant properties of the subspaces $P_{z}(i).$ Most of these properties can be established either ``by analyticity'' (i.e., they hold for real values of $\sigma$ and then by analytic continuation for complex values of $\sigma$ as well) or ``by continuity'' (i.e., we can verify by direct computation that they hold on the zero section and thus by continuity they hold also on a neighborhood of the zero section).

\begin{theorem}
[Behavior on the zero section]\label{zerosec:thm}For a point $(x,0)$ in the zero section, identify $T_{(x,0)}TM$ with $T_{x}M\oplus T_{x}M$ (tangent space to the zero section direct sum tangent space to the fiber). Then for all $\sigma\in\mathbb{R},$ $(\Phi_{\sigma})_{\ast}$ evaluated at $(x,0)$ is the linear map represented by the block matrix
\[
\left(
\begin{array}
[c]{cc}%
I & \sigma I\\
0 & I
\end{array}
\right)  .
\]
\end{theorem}

\begin{proof}
The composition property of the flow $\Phi_{\sigma}$ implies that the restriction of $(\Phi_{\sigma})_{\ast}$ to the tangent space at a fixed point $z$ is a one-parameter group of linear transformations of $T_{z}TM.$ This one-parameter group is then the exponential of some linear transformation $A.$ The value of $A$ is computed by differentiating $\left.  (\Phi_{\sigma})_{\ast}\right\vert _{T_{z}TM}$ with respect to $\sigma$ and evaluating at $\sigma=0.$ Since $\Phi_{\sigma}$ is smooth as a map of $TM\times\mathbb{R}$ into $TM,$ we can interchange the derivative with respect to $\sigma$ with the derivative in the space variable that computes the differential $(\Phi_{\sigma})_{\ast}.$ Differentiating $\Phi_{\sigma}$ with respect to $\sigma$ gives the vector field $X_{E}.$ Thus, working in local coordinates, $A$ will be the matrix of derivatives of the vector field $X_{E}$ evaluated at $z.$

If $q_{1},\ldots,q_{n},$ $p_{1},\ldots,p_{n}$ are local coordinates of the usual sort\footnote{We use the following sort of local coordinates on the \textit{co}-tangent bundle $T^*M.$ We let $q^{1},\ldots ,q^{n}$ be local coordinates on $M,$ and then we express a point in $T^{\ast }M$ as
\[
p_{j}dq^{j}
\]
(sum convention). Then $q^{1},\ldots ,q^{n},p_{1},\ldots ,p_{n}$ are our local coordinates on $T^{\ast }M.$ (If the usual coordinates on the tangent bundle are $q^{1},\ldots ,q^{n},v^{1},\ldots ,v^{n}$, then it is possible to view the $p_{j}$'s as the functions on the tangent bundle given by $p_{j}=g_{jk}v^{k}.$) In a coordinate system of this type, the canonical 2-form and Poisson bracket have simple forms and the Hamiltonian vector field associated to a function $f$ is
\[
X_{f}=\frac{\partial f}{\partial p_{j}}\frac{\partial }{\partial q^{j}}-%
\frac{\partial f}{\partial q^{j}}\frac{\partial }{\partial p_{j}}.
\]} on $TM\cong T^*M,$ then since the energy function $E$, when transferred to the cotangent bundle using the metric, is given by $E=\tfrac12 g^{ij}p_ip_j,$ we obtain
\[
X_{E}(q,p)=g^{jk}(q)p_{j}\frac{\partial}{\partial q^{k}}-\frac12 \frac{\partial g^{jk}}{\partial q^{l}}p_{j}p_{k}\frac{\partial}{\partial p_{l}}
\]
(sum convention). Differentiating and evaluating at $(q,0),$ we can see that $A$ is the block matrix
$\left(
\begin{array}
[c]{cc}%
0 & I\\
0 & 0
\end{array}
\right)  $ and so $\exp(\sigma A)$ is the matrix in the statement of the theorem.
\end{proof}

We now turn to an important scaling property of the subspaces $P_{z}(\sigma).$ This property reflects that the energy function $E$ is homogeneous of degree 2 in each fiber.

\begin{theorem}[Scaling]\label{scaling.thm}
For $c\in\mathbb{R},$ let $N_{c}:TM\rightarrow TM$ denote the scaling map in the fibers: $N_{c}(x,v)=(x,cv).$ Then for all real values of $\sigma,$ all nonzero values of $c,$ and all $z\in TM,$ we have
\[
P_{N_{c}z}(\sigma)=(N_{c})_{\ast}(P_{z}(c\sigma)).
\]
It follows that if the map $\sigma\rightarrow P_{z}(\sigma)$ has a holomorphic extension to a disk of radius $R$ then the map $\sigma\rightarrow P_{N_{c}z}(\sigma)$ has a holomorphic extension to a disk of radius $R/c.$
\end{theorem}

\begin{proof}
We make use of the relation $\Phi_{\sigma}=N_{c}\Phi_{c\sigma}N_{1/c},$ which holds because the energy function $E$ is homogeneous of degree 2 in each fiber.

To compute $P_{N_{c}z}(\sigma),$ we take the vertical tangent space $V$ at the point $\Phi_{-\sigma}(N_{c}z)$ and push it forward by $(\Phi_{\sigma})_{\ast}.$ But $\,$
\[
(\Phi_{\sigma})_{\ast}=(N_{c})_{\ast}(\Phi_{c\sigma})_{\ast}(N_{1/c})_{\ast},
\]
and $(N_{1/c})_{\ast}$ maps the vertical subspace at $\Phi_{-\sigma}(N_{c}z)$ to the vertical subspace $V^{\prime}$ at $N_{1/c}(\Phi_{-\sigma}(N_{c}z))=\Phi_{-c\sigma}(z).$ Thus,
\begin{align*}
P_{N_{c}z}(\sigma) &  =(N_{c})_{\ast}(\Phi_{c\sigma})_{\ast}(V^{\prime})\\
&  =(N_{c})_{\ast}(P_{z}(c\sigma)).
\end{align*}
\end{proof}

\begin{lemma}[Analyticity]\label{lemma:analytic}
Let $\mathcal{L}$ denote the (complex) Lagrangian Grassmannian bundle over $TM$ (with fiber $\mathcal{L}_{z}$). The map $P:TM\times\mathbb{R}\rightarrow\mathcal{L}$ given by (\ref{eqn:pushforwardV}) is analytic.
\end{lemma}

\begin{proof}
Choose a local analytic frame $v^{1},v^{2},\dots,v^{n}$ for the vertical distribution $V_{\cdot}$ over an open set $(\Phi_{-\sigma})(U)\subset TM$. Then $\{u^{j}(\sigma):=(\Phi_{\sigma})_{\ast}v^{j}\}$ is a local frame for $P(\sigma)$ over the open set $U\subset TM$. The geodesic flow is the Hamiltonian flow of the analytic function $E(x,v)=\frac{1}{2}g(v,v)$ and $TM$ is analytic, so by standard results for ordinary differential equations, the
geodesic flow $\Phi_{\sigma}$ depends analytically on $\sigma$ and on the initial conditions \cite[Prop 3.37]{Kohno}. (That is, $\Phi_{\sigma}$ is analytic as a map of $TM\times\mathbb{R}$ into $TM.$) It follows that the pushforward map $(\Phi_{\sigma})_{\ast}$ is also analytic, whence the frames $\{u_{z}^{j}(\sigma)\}$ depend analytically on $\sigma$ and $z$. The Lagrangian subspace spanned by these vectors then also depends analytically on $z$ and $\sigma.$ \end{proof}

\medskip

For each $z,$ the map $\sigma\rightarrow P_{z}(\sigma)$ has a holomorphic extension to \textit{some} disk. It then follows from the above scaling result that for all sufficiently small $c,$ the map $\sigma\rightarrow P_{N_{c}z}(\sigma)$ has a holomorphic extension to a disk of radius greater than 1. This observation is the key to the following existence result.

\begin{theorem}[Existence]\label{thm:continuation}
There exist $R>0$ and $\epsilon>0$ such that: (1) for each $z\in T^{R}M$ the map $\sigma\rightarrow P_{z}(\sigma)\in\mathcal{L}_{z}$ admits a holomorphic extension to a disk $D_{1+\epsilon}$ of radius $1+\epsilon$ in $\mathbb{C}$, and (2) the map $(z,\sigma)\mapsto P_z(\sigma)$ is real analytic as a map of $T^R M\times D_{1+\epsilon}$ into the Lagrangian Grassmannian bundle over $TM$.
\end{theorem}

\begin{proof}
In an analytic local coordinate system on $TM$ with origin at some $z_{0}$, we have the associated analytic local trivialization of the tangent bundle of $TM.$ Then we may think of the map $(z,\sigma)\rightarrow P_{z}(\sigma)$ as a map of $(U\subset\mathbb{R}^{2n})\times\mathbb{R}$ into the Grassmannian of $n$-dimensional complex subspaces of $\mathbb{C}^{2n}.$ By the lemma, this map is analytic. There exists, then, an open set $V$ in $\mathbb{C}^{2n}\times\mathbb{C}$ containing $U\times\mathbb{R}$ to which this map has a holomorphic extension. The set $V$ contains a set of the form $W_{1}\times W_{2},$ where $W_{1}$ is a neighborhood of the origin in $\mathbb{C}^{2n}$ and $W_{2}$ is a disk of some radius $R>0$ around the origin in $\mathbb{C}.$ This shows that there exists $R>0$ such that for all $z$ in some neighborhood of $z_{0},$ the map $\sigma\rightarrow P_{z}(\sigma)$ has a holomorphic extension to a disk of radius $R$ in $\mathbb{C},$ and the map $(z,\sigma)\mapsto P_z(\sigma)$ is real analytic. This amounts to saying that the maximum radius of such extensions is locally bounded away from zero.

Since $T^1 M$ is compact, the above argument shows that there is some minimum radius $\eta$ such that $P$ is defined and real analytic on $T^1 M\times D_\eta.$ For any $R$ with $0<R<\eta$, choose $\epsilon>0$ such that $R/\eta=1+\epsilon.$ Then by Theorem \ref{scaling.thm}, we obtain a map $P$ from $T^R M\times D_{1+\epsilon}$ into the Lagrangian Grassmannian as desired.
\end{proof}

\begin{theorem}[Integrability]\label{thm:integrable}
Suppose $R>0$ is such that there exists $\epsilon>0$ for which (1) for all $z\in T^R M$, the map $\sigma\rightarrow P_{z}(\sigma)\in\mathcal{L}_{z}$ has a holomorphic extension to a disk $D_{1+\epsilon}$, and (2) the map $(z,\sigma)\mapsto P_z(\sigma)$ is smooth as a map of $T^R M\times D_{1+\epsilon}$ into the Lagrangian Grassmannian bundle over $TM$. Then $P(i)$ is an integrable distribution on $T^{R}M.$
\end{theorem}

\begin{proof}
Let $\mathcal{K}_{n}$ denote the set of $n$-dimensional complex subspaces of $\mathbb{C}^{2n}.$ The group $GL(2n,\mathbb{C})$ acts transitively on $\mathcal{K}_{n}$ and the stabilizer subgroup of a point in $\mathcal{K}_{n}$ is isomorphic to the group $H$ of $2n\times2n$ matrices for which the lower left $n\times n$ block is zero. Because $H$ is a closed complex subgroup, $GL(2n,\mathbb{C})$ is a holomorphic fiber bundle over $\mathcal{K}_{n}$ with fiber $H.$ We can pick a local holomorphic section in a neighborhood of any point. This means that for $V_{0}\in\mathcal{K}_{n},$ there is a neighborhood $U$ of $V_{0}$ and a map $g(\cdot)$ from $U$ into $GL(2n,\mathbb{C})$ such that $V=g(V)(V_{0})$ for all $V\in U.$

Suppose $P(i)$ is not integrable, which means that there exist $z_{0}\in T^{R}M$ and vector fields $X$ and $Y$ lying in $P(i)$ with $[X,Y]_{z_{0}}\not \in P_{z_{0}}(i)$. Let $\alpha_{0}$ be the infimum of the set of $\alpha$ in $[0,1]$ for which there exist vector fields $X$ and $Y$ with $[X,Y]\not \in P_{z_{0}}(i\alpha)$.

Choose a local trivialization of $T^{\mathbb{C}}TM$, so that $P_{z}(\sigma)$ may be thought of as a subspace of $\mathbb{C}^{2n}$ which depends analytically on $(z,\sigma)$, and let $V_{0}:=P_{z_{0}}(0)=V_{z_{0}}^{\mathbb{C}}TM$. For $z$ in some neighborhood of $z_{0}$ and for $\sigma$ in some neighborhood of $0$, $P_{z}(\sigma)$ will belong to the domain of the function $g$ in the first paragraph. Choose a basis $u^{1},\dots,u^{n}$ of $P_{z_{0}}(0)$ and set $u^{j}(z,\sigma):=g(P_{z}(\sigma))u^{j}$ for $j=1,\dots,n$. Then for each $j$, $u^{j}(z,\sigma)$ is real analytic in $z$ and $\sigma$, and $\{u^{j}(z,\sigma)\}_{j=1,\dots,n}$ is a basis for $P_{z}(\sigma)$. Hence, for each $1\leq j,k\leq n$,
\[
\lbrack u^{j}(z,\sigma),u^{k}(z,\sigma)]\wedge u^{1}(z,\sigma)\wedge \cdots\wedge u^{n}(z,\sigma)
\]
is also real analytic. When $\sigma$ is real, $P_{z}(\sigma)$ is involutive since it is the image under a diffeomorphism of an involutive distribution, whence the above quantity is zero. We conclude then that the above quantity, for each $1\leq j,k\leq n$, is zero for $\sigma$ in a neighborhood of $0$ in $\mathbb{C}$; that is, that $P_{z}(i\alpha)$ is involutive for $\alpha$ near $0$, so that $\alpha_{0}$ is necessarily positive.

Next, we repeat the above argument, but centered on $(z_{0},i\alpha_{0})$. That is, for $z$ in some neighborhood of $z_{0}$ and for $\sigma$ in a neighborhood of $i\alpha_{0}$, the family of subspaces $P_{z}(\sigma)$ (still working in the real-analytic trivialization of $T^{\mathbb{C}}TM$) will be in the domain of the function $g$ of the first paragraph. Choose a basis $u^{1},\dots,u^{n}$ of $V_{0}:=P_{z_{0}}(i\alpha_{0})$ and define
$u^{j}(z,\sigma):=g(P_{z}(\sigma))u^{j}.$ Then by the definition of $\alpha_{0}$,
\begin{equation}
\lbrack u^{j}(z,i\alpha),u^{k}(z,i\alpha)]\wedge u^{1}(z,i\alpha)\wedge\cdots\wedge u^{n}(z,i\alpha)=0\label{eqn:int-wedge}
\end{equation}
for each $\alpha<a_{0}$ and for each $1\leq j,k\leq n$. Since $u^{j}(z,\sigma)$ is holomorphic in $\sigma$, the above equality must also be true for $\alpha=\alpha_{0}.$ Since $\{u^{j}(z,i\alpha_{0})\}_{j=1,\dots,n}$ is a basis for $V_{0},$ we see that by taking linear combinations of (\ref{eqn:int-wedge}), we obtain
\[
\lbrack X,Y]\wedge u^{1}(z,i\alpha)\wedge\cdots\wedge u^{n}(z,i\alpha)=0
\]
for all $X,Y\in P_{z_{0}}(i\alpha_{0})$, which contradicts the definition of $\alpha_{0}.$
\end{proof}

\begin{theorem}
[K\"{a}hler structure]\label{thm:pos Lagr}In a neighborhood of the zero section, $P_{z}(i)$ intersects its complex conjugate only at zero. Thus there is some $R^{\prime}>0$ (possibly smaller than the radius in the existence theorem) such that for $z\in T^{R^{\prime}}M$, the distribution $P_{z}(i)$ is the $(1,0)$-tangent space of a complex structure.
\par
Moreover, there exists some $R\in(0,R^{\prime}]$ such that for $z\in T^{R}M$, the distribution $P_{z}(i)$ is positive with respect to the canonical symplectic form, that is, $-i\omega(Z,\overline{Z})>0$ for every $Z\in P_{z}(i)$. Hence, $P_{z}(i)$ defines a K\"{a}hler structure on $T^{R}M$.
\end{theorem}

\begin{proof}
Choose a local analytic frame $\{X_{j}\}_{j=1}^{n}$ of $P(i)$ over an open set $U_{\alpha}\subset TM$ which intersects the zero section in the (nonempty, open in $M$) set $U_{0}=M\cap U_{\alpha}$. On the zero section, we know that $P(i)$ looks like the Euclidean $(1,0)$-bundle; in particular, $X_{j}(p)-\overline{X}_{j}(p)\neq0$ for every $p\in U_{0}$. Since the vector fields $X_{j}-\overline{X}_{j}$ are analytic (hence continuous), it follows that $X_{j}(p)-\overline{X}_{j}(p)\neq0$ for all $p$ in some open (in $TM$) neighborhood $V$ with $U_{0}\subset V\subset U_{\alpha}.$

Do this for each open set $U_{\alpha}$ intersecting the zero section, then take a maximal tube inside the union of the $V$s (we only need to consider a finite set of $U_{\alpha}$s, and hence of $V$s, since $M$ is compact) to obtain $T^{R_{0}}M$.

To see that $P_{z}(i)$ is positive, we work in the same local trivialization. On the zero section, the distribution $P_{(x,0)}(i)$ is just that of the standard Euclidean complex structure on $\mathbb{R}^{2n}$, so that $-i\,\omega_{(x,0)}(X_{j},\overline{X}_{j})>0.$ But $-i\,\omega_{z}(X_{j},\overline{X}_{j})$ is a continuous (in fact analytic) function of $z$, so it must be positive in a neighborhood of the zero section. Again, by the compactness of $M$ we obtain a neighborhood of the zero section on which $P_{z}(i)$ is positive. Combined with the fact that the distribution $P(i)$ is Lagrangian, this shows that the complex structure defined by the distribution $P(i)$ is $\omega$-compatible, whence $T^{R}M$ is K\"{a}hler.
\end{proof}

\begin{theorem}
[Adaptedness]\label{adapted.thm}Let $R$ be any positive number such that $P_{z}(i)$ exists (in the sense of Theorem \ref{thm:continuation}) for all $z\in T^{R}M$ and has trivial intersection with its complex conjugate. Then the associated complex structure is ``adapted'' in the sense of Lempert--Sz\H{o}ke (Definition 4.1 in \cite{Lempert-Szoke}).
\end{theorem}

\begin{remark}
In Section \ref{jacobi.sec}, we show $P_{z}(i)$ can be computed in terms of Jacobi fields in a way that agrees with formulas in \cite{Lempert-Szoke} for the adapted complex structure.
\end{remark}

\begin{proof}
The main point is that the geodesic flow $\Phi_{t}$ leaves invariant the tangent bundle of a geodesic, as a surface inside $TM,$ and the action of $\Phi_{t}$ on this surface is the same as the geodesic flow for the real line.

Now, the geodesic flow on $T^{R}M$ leaves each strip inside $T^{R}M$ of the form $\Psi_{z}(S_{R})$ invariant. In fact, for each $t\in\mathbb{R},$ the time-$t$ geodesic flow on such a strip corresponds to the map $(\sigma,\tau)\rightarrow(\sigma+t\tau,\tau).$ (That is to say, $\Phi_{t}(\Psi_{z}(\sigma,\tau))=\Psi_{z}(\sigma+t\tau,\tau).$) Furthermore, for each $\sigma,$ the curve $s\rightarrow\Psi_{z}(\sigma,\tau+s)$ is ``vertical,'' i.e., lying in a single fiber. It follows that the vector
\begin{equation}
\frac{\partial}{\partial\tau}\Psi_{z}(\sigma+t\tau,\tau)=(\Psi_{z})_{\ast} (\partial/\partial\tau)+t(\Psi_{z})_{\ast}(\partial/\partial\sigma)
\label{lin_comb}
\end{equation}
belongs to $P_{u}(t)$ for all real numbers $t$ and each point $u$ in the strip. (Every point in the strip is of the form $\Psi_{z}(\sigma+t\tau,\tau)$ for some $\sigma$ and $\tau.$)

Now, the right-hand side of (\ref{lin_comb}) is clearly a holomorphic vector-valued function of $t$ for $t\in\mathbb{C}.$ The family $P_{u}(t)$ of subspaces also depends holomorphically on $t$ for $t$ in a ball of radius greater than $1$ around the origin. It follows that the right-hand side of (\ref{lin_comb}) is in $P_{u}(t)$ for all $t$ in that ball, including $t=i.$ This means that $(\Psi_{z})_{\ast}(\partial/\partial\tau)+i(\Psi_{z})_{\ast}(\partial/\partial\sigma)$ belongs to $P_{u}(i).$ Recalling that $P_{u}(i)$ is the $J=i$ subspace and $\overline{P_{u}(i)}$ is the $J=-i$ subspace, a little algebra shows that $J((\Psi_{z})_{\ast}(\partial/\partial\sigma))=(\Psi_{z})_{\ast}(\partial/\partial\tau).$ This shows that $\Psi_{z}$ is holomorphic.
\end{proof}

\subsection{Involution and K\"{a}hler potential\label{kahler.sec}}

We next give proofs from our point of view that twice the energy function $E$ is a K\"{a}hler potential for the adapted complex structure, and that the map $(x,v)\longmapsto(x,-v)$ is antiholomorphic. As mentioned in Section \ref{sec:main results}, these results already appear in \cite{Lempert-Szoke}, and show that the complex structure we have constructed satisfies the conditions of the theorem on p. 568 of \cite{Guillemin-Stenzel1}. Scaling (or the homogeneity property of $E$) gets used in both cases.

\begin{theorem}
[Involution]\label{involution.thm}Let $R$ be as in Theorem \ref{thm:continuation}. Then the map $(x,v)\longmapsto(x,-v)$ is antiholomorphic with respect to the adapted complex structure on $T^{R}M.$
\end{theorem}

\begin{proof}
Because $P(\sigma)=\overline{P(\sigma)}$ for $\sigma\in\mathbb{R},$ it is easily seen that $P(\sigma-i\tau)=\overline{P(\sigma+i\tau)}.$ By Theorem \ref{scaling.thm}, we have $(N_{-1})_{\ast}P(\sigma)=P(-\sigma)$. On $T^{R}M$, analytically continuing and evaluating at $\sigma=i$ yields $(N_{-1})_{\ast}P(i)=P(-i)=\overline{P(i)}.$ This shows that $N_{-1}$ (i.e., the map $(x,v)\longmapsto(x,-v)$) is antiholomorphic.
\end{proof}

\begin{theorem}[K\"{a}hler potential]
\label{thm:kaehler_potential}
The function $\kappa(x,v):=g(v,v)=2E(x,v)$ is a K\"{a}hler potential for the adapted complex structure on $T^{R}M$. Specifically, $\operatorname{Im}\bar{\partial}\kappa=\Theta\,$ where $\Theta$ is the canonical 1-form on $TM\cong T^{\ast}M.$
\end{theorem}

We begin with a lemma.

\begin{lemma}
\label{lemma:theta_comp}
For all $Z\in P_{z}(\sigma)\subset T_{z}^{\mathbb{C}}TM$ we have
\begin{equation}
\Theta(Z)=\sigma Z(E). \label{theta.id}
\end{equation}
\end{lemma}

\begin{proof}
Recall that $\Theta$ is the canonical 1-form on $TM$ induced from that on $T^*M$ by the identification via the metric $g$; in particular, it is given by $\Theta_{(x,v)}(Z)=g(v,\pi_{\ast}Z).$ We begin by proving that for all $\sigma$ in any disk around the origin in $\mathbb{C}$ on which $P_{z}(\sigma)$ is defined holomorphically. This is the same as saying that the 1-forms $\Theta$ and $\sigma dE$ agree on vectors in $P(\sigma).$ We first establish this when $\sigma$ is in $\mathbb{R}.$ Suppose $X$ is a vector field lying in the vertical tangent space at each point in $TM,$ and let $X^{\sigma}=(\Phi_{\sigma})_{\ast}(X).$ We wish to show that the function
\[
u(z,\sigma):=\Theta(X^{\sigma})-\sigma X^{\sigma}(E)
\]
is identically zero for all $\sigma.$ Let us differentiate with respect to $\sigma$. Since $dX^{\sigma}/d\sigma=-[X_{E},X^{\sigma}]$ and $X_{E}(E)=0,$ we obtain
\begin{align*}
\frac{\partial u}{\partial\sigma}  &  =\Theta([X^\sigma,X_E])-X^\sigma(E)+\sigma X_E X^\sigma(E))\\
&  =X^\sigma(\Theta(X_E))-X_E(\Theta(X^\sigma))-d\Theta(X^{\sigma},X_{E})-X^{\sigma}(E)+\sigma X_{E}(X^{\sigma}(E))\\
&  =-X_E(u)+2X_\sigma(E)-\omega(X_E,X^\sigma)-X^\sigma(E)\\
&  =-X_{E}(u),
\end{align*}
where in the second line we use the identity $d\alpha(X,Y)=X(\alpha(Y))-Y(\alpha(X))-\alpha([X,Y]),$ in the third line we use the computation $\Theta(X_{E})=2E$ and the definition of $\omega$, and in the fourth line we use the definition of the Hamiltonian vector field to obtain $-d\Theta(X^{\sigma},X_{E})=-dE(X^{\sigma})=-X^{\sigma}(E)$ from which the result follows.

Now, the solution to the equation $\partial u/\partial\sigma=-X_{E}u$ is just $u(z,\sigma)=f(\Phi_{-\sigma}(z)),$ where $f(z)=u(z,0).$ In our case, $u(z,0)=0,$ because $\Theta$ is zero on the vertical subspace, and so $u$ is identically equal to zero. Since $X^{\sigma}$ can take any value in $P(\sigma)$ (for appropriate choice of $X$), we conclude that (\ref{theta.id}) holds for all real $\sigma.$ Using an argument similar to the proof of Theorem \ref{thm:integrable}, we see that (\ref{theta.id}) holds for all $\sigma\in D_{1+\epsilon}.$
\end{proof}

\bigskip

\begin{proof}[Proof of Theorem \ref{thm:kaehler_potential}]
We start with the two facts that $d=\partial+\bar{\partial}$ and $\omega$ is of type $(1,1)$. Then
\begin{align}
\omega &  =-d\Theta=-(\partial+\bar{\partial})(\Theta^{1,0}+\Theta^{0,1})\nonumber\\
&  =-\partial\Theta^{0,1}-\bar{\partial}\Theta^{1,0} \label{eqn:sympprim}
\end{align}
since $\partial\Theta^{1,0}=\bar{\partial}\Theta^{0,1}=0.$

If we now apply (\ref{theta.id}) with $\sigma=i$ and $\sigma=-i$ (for $z\in T^{R}M$), we obtain
\begin{align*}
\Theta^{1,0}(X)  &  =\Theta(X^{1,0})=iX^{1,0}(E)=i\partial E(X)\text{, and}\\
\Theta^{0,1}(X)  &  =\Theta(X^{0,1})=-iX^{0,1}(E)=-i\bar{\partial}E(X).
\end{align*}
Hence,
\[
\Theta=\Theta^{1,0}+\Theta^{0,1}=i\partial E-i\bar{\partial}E =2\operatorname{Im}\bar{\partial}E=\operatorname{Im}\bar{\partial}\kappa.
\]
\end{proof}

\subsection{Computation of the adapted complex structure in terms of Jacobi
fields\label{jacobi.sec}}

We now give a more-or-less explicit way of computing the adapted complex structure. The results of this subsection also give a more direct way of verifying that the complex structure we have defined coincides with the adapted complex structure of Lempert and Sz\H{o}ke.

Fix a point $z=(x,v)$ in $TM.$ Choose a basis $\{v_{j}\}$ for $T_{x}M$ and let $\xi_{j}$ and $\eta_{j}$ denote the horizontal and vertical lifts, respectively, of $v_{j}$ to $T_{z}TM.$ We now push these vectors forward by the geodesic flow, defining a family of tangent vectors along the curve $\Phi_{\sigma}(z)$:
\begin{align*}
\xi_{j}(\sigma)  &  :=(\Phi_{\sigma})_{\ast}(\xi_{j})\\
\eta_{j}(\sigma)  &  :=(\Phi_{\sigma})_{\ast}(\eta_{j}).
\end{align*}
Finally, we project these vectors down to $M$ by defining
\begin{align*}
v_{j}(\sigma)  &  :=\pi_{\ast}(\xi_{j}(\sigma))\\
w_{j}(\sigma)  &  :=\pi_{\ast}(\eta_{j}(\sigma)).
\end{align*}
The vector fields $v_{j}$ and $w_{j}$, defined along the geodesic $\pi(\Phi_{\sigma}(x,v))$, are Jacobi fields. (We refer the reader to \cite[Chap. 4]{Jost} for background on Jacobi fields.) Note that $\xi_{j}(0)=\xi_{j},$ $\eta_{j}(0)=\eta_{j},$ $v_{j}(0)=v_{j},$ and $w_{j}(0)=0.$

For all sufficiently small $\sigma,$ $\{v_{j}(\sigma)\}$ is a basis for the tangent space to $M$ at the point $\pi(\Phi_{\sigma}(z))$ and $\{\xi_{j}(\sigma),\eta_{j}(\sigma)\}$ is a basis for the tangent space to $TM$ at the point $\Phi_{\sigma}(z).$ We let $f_{z}(\sigma)$ denote the matrix expressing $w_{j}(\sigma)$ in terms of $v_{j}(\sigma)$:
\begin{equation}
w_{j}(\sigma)=\sum\nolimits_{k}f_{z}(\sigma)_{j}^{k}v_{k}(\sigma).
\label{fdef}
\end{equation}
It then follows that
\[
\pi_{\ast}\left[  \eta_{j}(\sigma)-\sum\nolimits_{k}f_{z}(\sigma)_{j}^{k}
\xi_{k}(\sigma)\right]  =0.
\]
This says that the vectors $\eta_{j}(\sigma)-\sum f_{v}(\sigma)_{j}^{k}\xi_{k}(\sigma)$ are contained in the vertical subspace of $T_{\Phi_{\sigma}(z)}TM,$ and by the independence of $\xi,\eta,$ these vectors actually span the vertical subspace.

Note that from the way $\xi_{j}(\sigma)$ and $\eta_{j}(\sigma)$ are defined, we have $(\Phi_{\sigma})_{\ast}\xi_{j}(-\sigma)=\xi_{j}$ and $(\Phi_{\sigma})_{\ast}\eta_{j}(-\sigma)=\eta_{j}.$ Thus the vectors
\begin{equation}
\eta_{j}-\sum\nolimits_{k}f_{z}(-\sigma)_{j}^{k}\xi_{k}\in T_{z}TM
\label{pSpan}
\end{equation}
span the push-forward (to $T_{z}TM$) of the vertical subspace of $T_{\Phi_{-\sigma}(z)}TM.$

It follows that if, for a fixed $z,$ the matrix-valued function $f_{z}(\cdot)$ admits an analytic continuation to a disk of radius greater than $1,$ then the family of subspaces $P_{z}(\sigma)$ also admits an analytic continuation,
defined by putting in a complex value for $\sigma$ into (\ref{pSpan}) and taking the span of the resulting vectors. (Since the $\eta$'s are independent of the $\xi$'s, any collection of $n$ vectors of the form (\ref{pSpan}) are linearly independent.) We may encapsulate the preceding observations in the following theorem.

\begin{theorem}
For a given $z\in TM,$ let $f(\cdot)$ be the matrix-valued function defined, for sufficiently small $\sigma\in\mathbb{R},$ by (\ref{fdef}). If $f$ admits a matrix-valued holomorphic extension to a disk of radius $R,$ then so also do the subspaces $P_{z}(\cdot),$ by setting $P_{z}(\sigma)$ equal to the span of the vectors in (\ref{pSpan}).
\end{theorem}

We can demonstrate directly, using the matrix $f$, that the complex structure arising from the time-$i$ geodesic flow coincides with the adapted complex structure as defined by Lempert--Sz\H{o}ke \cite{Lempert-Szoke}.

Lempert and Sz\H{o}ke show that for $v\in T^{R}M$, the adapted complex tensor $J_{v}\in End(T_{v}TM)$ is specified by \cite[Eq. (5.8)]{Lempert-Szoke}
\begin{equation}
J_{v}\xi_{j}=\left(  \left(  \operatorname{Im}f(i)\right)  ^{-1}\right)
_{~j}^{k}\left(  \eta_{k}-\left(  \operatorname{Re}f_{~k}^{l}(i)\right)
\xi_{l}\right)  ,~j=1,\dots,n. \label{eqn:L-S J tensor}%
\end{equation}
Since $\{\xi_{j}\}_{j=1}^{n}$ is linearly independent, the $(1,0)$-tangent space $T_{v}^{1,0}TM$ is spanned by~$\{\xi_{j}^{1,0}\}_{j=1}^{n}$ where the projection onto the $(1,0)$-tangent space is given by
\[
\xi^{1,0}:=\frac{1}{2}(1-iJ)\xi.
\]
Writing $f(i)=f_{1}+if_{2}$ (i.e., $f_{1}=\operatorname{Re}f(i)$ and $f_{2}=\operatorname{Im}f(i)$), it follows from (\ref{eqn:L-S J tensor}) that
\begin{equation}
\xi_{j}^{1,0}=\frac{i}{2}\sum\nolimits_{k}\left[  \left(  \overline{f(i)} f_{2}^{-1}\right)  _{~j}^{k}\xi_{k}~-~\left(  f_{2}^{-1}\right)_{~j}^{k}
\eta_{k}\right]. \label{eqn:1,0-components}
\end{equation}
The $(1,0)$-tangent space is therefore spanned by the vectors $\sum\nolimits_{k}\left[(\overline{f(i)}f_{2}^{-1})_{\,~j}^{k}\xi_{k}-(f_{2}^{-1})_{~j}^{k}\eta_{k}\right].$

As shown by Lempert and Sz\H{o}ke \cite[Lemma 6.7]{Lempert-Szoke}, $\operatorname{Im}f(i)=f_{2}$ is positive definite\footnote{This is actually a necessary step to derive the expression (\ref{eqn:L-S J tensor}); it can be verified with a short computation using the fact that $-i\omega(\xi^{1,0},\xi^{0,1})>0$ (Theorem \ref{thm:pos Lagr})}, and is hence invertible. Thus, the $(1,0)$-tangent space of the adapted complex structure, as defined by Lempert--Sz\H{o}ke, is spanned by the vectors $\sum\nolimits_{k}$ $\overline{f(i)}_{~j}^{k}\xi_{k}-\eta_{j}$; that is, it is $P(i)$.

\subsection{Connections to geometric quantization\label{quant.sec}}

The arguments we have used in this section also show that for any $\tau>0,$ we can obtain a well-defined K\"{a}hler distribution on some $T^{R}M$ by considering the subspace $P_{z}(i\tau)$ at each point $z\in T^{R}M.$ In light of our scaling result, Theorem \ref{scaling.thm}, this is equivalent to taking the adapted complex structure and rescaling by a factor of $\tau$ in the fibers. (That is, $P_{z}(i\tau)$ is the same as $(N_{1/\tau})_{\ast}P_{N_{\tau}z}(i).$) We obtain, then, a one-parameter family $J_{\tau}$ of complex structures indexed by the positive real number $\tau.$ If $f$ is a holomorphic function on $T^{R}M$ with respect to the adapted complex structure, then the function $f_{\tau}(x,v):=f(x,\tau v)$ is holomorphic with respect to $J_{\tau}.$ Note that in the limit as $\tau$ approaches zero, $f_{\tau}$ becomes constant along the fibers of $TM.$ This reflects the idea that the Lagrangian distribution associated to $J_{\tau}$ is obtained from the vertical distribution by the time-$i\tau$ geodesic flow, so that the$\ J_{\tau}$-distribution converges to the vertical distribution as $\tau$ tends to zero.

Now, given \textit{any} complex structure on (a neighborhood of the zero section in) $TM,$ one can rescale by a constant in the fibers to obtain a one-parameter family of complex structures. Furthermore, the distribution associated to this family of complex structures will converge to the vertical distribution as $\tau$ tends to zero. What is interesting in the case we are considering is that the one-parameter family of complex structures obtained by scaling the adapted complex structure can also be obtained by \textit{starting} with the vertical distribution and applying the imaginary-time geodesic flow.

The procedure of geometric quantization associates to each integrable Lagrangian distribution, or \textit{polarization}, the Hilbert space of sections of a\ certain line bundle\footnote{In geometric quantization, one studies \textit{prequantum} line bundles, that is, complex Hermitian line bundles with connection with curvature $-i\omega$.} which are covariantly constant along the distribution.

In \cite{Hall02}, the first author has considered geometric quantization on $TM\cong T^{\ast}M$ in the case that $M$ is a compact Lie group $K$ with a bi-invariant metric. The paper \cite{Hall02} considers the pairing map between the vertically polarized Hilbert space and the K\"{a}hler-polarized Hilbert space (with half-forms) associated to the adapted complex structure. It turns out that this pairing map is unitary (up to a constant) and coincides (up to a constant) with the generalized Segal--Bargmann transform, which was introduced in \cite{Hall94,Hall97} and defined in terms of the heat equation on $K.$ This result is surprising because the procedures of geometric quantization apparently have nothing to do with the heat equation.

In \cite{Florentino-Matias-Mourao-Nunes05,Florentino-Matias-Mourao-Nunes06}, Florentino--Mattias--Mour\~{a}o--Nunes looked at the results of \cite{Hall02} in terms of the one-parameter family of complex structures described above, namely those obtained from the adapted complex structure by scaling in the fibers. Using ideas similar to the ones in \cite{Axelrod-DellaPietra-Witten}, these authors consider a parallel transport in the Hilbert bundle associated to the family of complex structures. This means that the base of their bundle is the positive half-line and the fiber over a point $\tau$ is the geometric quantization Hilbert space associated to the complex structure $J_{\tau}.$ They compute this parallel transport and show that it is given in terms of the heat equation on $K.$ This calculation goes a long ways toward clarifying the results of \cite{Hall02}.

The present paper adds one more clarifying insight to the picture: the one-parameter family of complex structures that Florentino--Mattias--Mour\~{a}o--Nunes are considering are all obtained from the vertical distribution by means of the imaginary-time geodesic flow. Since the ``transport'' at the classical level (the Lagrangian distributions) is given in terms of the imaginary-time geodesic flow, it is not surprising that the transport at the quantum level (the Hilbert spaces) is given in terms of the heat equation. (The quantization of the energy function $E(x,v)=\frac{1}{2}g(v,v)$, by whatever quantization procedure one prefers, generally comes out to be the Laplacian plus a multiple of the scalar curvature, where in the case of a compact group with a bi-invariant metric, the scalar curvature is a constant.)

Thiemann makes a similar point in \cite{Thiemann96,Thiemann1}. He proposes that the transition from functions of the position variables to the holomorphic functions can be made by using the imaginary-time geodesic flow (as explained in Section \ref{sec:Taylor series}). He then argues that the quantum counterpart of this transition (the Segal--Bargmann transform) should be achieved by the quantum counterpart of the imaginary time geodesic flow, namely, the heat semigroup. Actually, Thiemann proposes that \textit{any} (sufficiently regular) Hamiltonian flow can be used in place of the geodesic flow, though very few examples have been considered so far. In a future paper, we hope to look at the Hamiltonian flow associated to a charged particle in a magnetic field.

Meanwhile, J. Huebschmann \cite{Huebschmann} has examined the results of \cite{Hall02} from the point of view of the Kirillov character formula. In Huebschmann's approach, the heat equation enters because the characters are eigenfunctions for the Laplacian.

\section{Holomorphic functions \label{sec:Taylor series}}

In the previous section, we looked at the action of the imaginary-time geodesic flow on the vertical distribution. In this section, we look at the action of the imaginary-time geodesic flow on functions. Associated to any distribution $P,$ there is a naturally associated class of functions, namely, those functions $f$ such that for every $z$ and every $X\in\overline{P_{z}},$ $Xf=0.$ (It is customary to consider functions that are constant in the directions of $\overline{P}$ rather than $P.$) In the case of the vertical distribution $\overline{P}=P$ and the functions constant in the $P$-directions are just the functions that are constant along each fiber of the tangent bundle. In the case of the distribution $P(i),$ the functions with derivative zero in the directions of $\overline{P(i)}$ are the holomorphic functions with respect to the adapted complex structure. We would like to see that these two classes of functions are related by the time-$i$ geodesic flow.

Suppose $\psi$ is a function whose derivatives in the vertical directions are zero. Then it is easily seen that the derivatives of $\psi\circ\Phi_{\sigma}$ in the directions of $P(-\sigma)$ are zero. If we formally set $\sigma=i,$ we see that the derivatives of $\psi\circ\Phi_{i}$ should be zero in the direction of $P(-i)=\overline{P(i)}.$ (Note that $P(\sigma)$ is real for real $\sigma,$ so that $P(\sigma-i\tau)=\overline{P(\sigma+i\tau)}.$) The conclusion is that if $\psi$ is constant along the leaves of the vertical distribution, then we expect $\psi\circ\Phi_{i}$ (suitably interpreted) to be holomorphic with respect to the adapted complex structure on $T^{R}M.$ In this section we will fulfill that expectation and give three different but equivalent ways of interpreting the expression $\psi\circ\Phi_{i}.$

If $\psi$ is constant along the vertical distribution, then $\psi=f\circ\pi$ for some function $f$ on $M.$ We wish to compose with $\Phi_{\sigma}$ and then ``set $\sigma=i$'' to obtain the function
\begin{equation}
f(\pi(\Phi_{i}(z))). \label{f_imaginary}
\end{equation}
The first way to interpret (\ref{f_imaginary}) is to look at the map
\begin{equation}
\sigma\rightarrow f(\pi(\Phi_{\sigma}(z))) \label{extend1}
\end{equation}
for a fixed $f$ and $z.$ This function is well defined for all real $\sigma,$ and we can attempt to analytically continue it to a ball of radius greater than $1$ about the origin in the complex plane. If such an analytic continuation exists, the value at $\sigma=i$ can but understood as the value of
(\ref{f_imaginary}).

The second way to interpret (\ref{f_imaginary}) is to think of the action of the geodesic flow on functions as (formally) the exponential of the Hamiltonian vector field $X_{E}.$ This approach is the one proposed by Thiemann in \cite{Thiemann96}. Since
\[
\frac{d}{d\sigma}f\circ\pi\circ\Phi_{\sigma}=X_{E}(f\circ\pi\circ\Phi_{\sigma
}),
\]
we have, at least formally,
\[
f\circ\pi\circ\Phi_{\sigma}=e^{\sigma X_{E}}(f\circ\pi),
\]
where
\begin{align*}
e^{\sigma X_{E}}\phi &  =\sum_{k=0}^{\infty}\frac{\sigma^{k}}{k!}X_{E}^{k}\phi\\
&  =\sum_{k=0}^{\infty}\frac{\sigma^{k}}{k!}\{E,\{E,\ldots\{E,\phi \}\ldots\}\}.
\end{align*}
This is Thiemann's point of view. (See also \cite{Hall-Mitchell02}, where this method is worked in a very explicit way in the case that $M$ is a sphere.)

The third way to interpret (\ref{f_imaginary}) is to observe that
\[
\pi(\Phi_{\sigma}(x,v))=\exp_{x}(\sigma v),
\]
where $\exp_{x}:T_{x}M\rightarrow M$ is the geometric exponential map. Suppose we embed $M$ in a real-analytic fashion as a totally real submanifold of some complex manifold $X$ of complex dimension $n$, as in work of Bruhat--Whitney \cite{Bruhat-Whitney} and Grauert \cite{Grauert58}. Then for each $x,$ the exponential map can be extended to a holomorphic map of a ball around the origin in $T_{x}^{\mathbb{C}}M$ into $X.$ If $f$ is real-analytic on $M,$ then it has a holomorphic extension $f_{\mathbb{C}}$ to a neighborhood of $M$ in $X$. Thus, the analytic continuation in $\sigma$ of (\ref{extend1}) can be accomplished by setting $f(\pi(\Phi_{\sigma}(z)))$ equal to $f_{\mathbb{C}}(\exp_{x}(\sigma v))$ for $\sigma$ in a small enough ball in $\mathbb{C}$ that $\sigma v\in T_{x}^{\mathbb{C}}M$ will lie in the domain of the analytic continuation of $\exp_{x}.$ Putting $t=i$ gives \begin{equation}
f(\pi(\Phi_{i}(z)))=f_{\mathbb{C}}(\exp_{x}(iv)), \label{exp_analytic}
\end{equation}
where on the right-hand side of (\ref{exp_analytic}), $\exp_{x}(iv)$ refers to the analytically continued exponential map.

What (\ref{exp_analytic}) is really saying is that we should identify $T^{R}M$ with a neighborhood of $M$ in $X$ by means of the map $(x,v)\rightarrow\exp_{x}(iv).$ (Thus, the value of a holomorphic function on $T^{R}M$ at $(x,v)$ is obtained by evaluating a holomorphic function on $X$ at $\exp_{x}(iv).$) This identification is implicit in Section 5 of \cite{Guillemin-Stenzel2} and was made explicit in a personal communication of Stenzel with the first author. The same identification is also used in the work of S. Halverscheid \cite{Halverscheid01}.

We now provide theorems showing that the holomorphic functions on $T^{R}M$ with respect to the adapted complex structure can indeed be obtained by means of the imaginary-time geodesic flow, using any one of the three interpretations discussed above. As mentioned in Section \ref{sec:main results}, the first part of the Theorem already appears in \cite{Lempert-Szoke}[Prop. 3.2], although our proof is independent.

\begin{theorem}
\label{thm:taylor}Let $R$ be as in Theorem \ref{thm:acs existence} and suppose a real-analytic function $f:M\rightarrow\mathbb{C}$ admits an analytic continuation $f_{\mathbb{C}}:T^{r}M\rightarrow\mathbb{C}$ for some $0<r\leq R.$ Then for each $z\in T^{r}M$, the map
\[
\sigma\mapsto\left(  f\circ\pi\circ\Phi_{\sigma}\right)  (z)
\]
admits an analytic continuation in $\sigma$ from $\mathbb{R}$ to $\{\sigma+i\tau\in\mathbb{C}:\left\vert \tau\right\vert <r/\left\vert v\right\vert \}.$ Moreover, the function $f_{\mathbb{C}}$ is given by the value of the continuation at $\sigma=i$,
\[
f_{\mathbb{C}}(z)=\left(  f\circ\pi\circ\Phi_{i}(z)\right)  ,
\]
and can be expressed as an absolutely convergent power series
\begin{equation}
f_{\mathbb{C}}(z)=\sum_{k=0}^{\infty}\frac{i^{k}}{k!}\left(  X_{E}^{k}(f\circ\pi)\right)  (z). \label{eqn:Taylor series}
\end{equation}
\end{theorem}

\begin{remark}
The function $X_{E}^{k}(f\circ\pi)$ can be written $\{E,\dots,\{E,f\circ\pi\}\dots\}$ ($k$ Poisson brackets). A computation then shows that, for each fixed $x\in M,$ the function  $v\rightarrow X_{E}^{k}(f\circ\pi)(x,v)$ is a homogeneous polynomial of degree $k.$ This means that (\ref{eqn:Taylor series}) may be thought of as the real-variable Taylor series of $f_{\mathbb{C}}$ in the fibers. (The analogous construction in the case of the real line would be the expansion of a holomorphic function $f(x+iy)$ as a power series in $y$ for each fixed $x.$)
\end{remark}

\begin{remark}
Building on work of L. Boutet de Monvel, Guillemin and Stenzel \cite[Thm. 5.2]{Guillemin-Stenzel2} have shown that the following result holds for all sufficiently small $R$: A function $f$ on $M$ admits an analytic continuation to $T^{R}M$ that is smooth up to the boundary if and only if $f$ is of the form $\exp(-RP)g$ for some smooth function $g.$ Here $P=\sqrt{\Delta},$ where $\Delta$ is the (positive) Laplacian. This condition on $f$ implies that the Taylor series expansion of $\exp(RP),$ when applied to $f,$ converges (in, say, the sup norm). The convergence on $T^{R}M$ of the Taylor series expansion of $\exp(iX_{E})$, when applied to $f\circ\pi,$ is an analogous result in our approach to the subject.
\end{remark}

The following result provides a way to interpret the expression $\exp_x(iv);$ in particular, it gives a formula for the holomorphic extension of a function when such an extension exists. The result appears already in the proof of Proposition 3.2 in \cite{Szoke} (see also Theorem 3.4 in \cite{Szoke95}), and is also implicitly contained in Section 5 of \cite{Guillemin-Stenzel2}, in the assertion on p. 638 that the tubes $M_{\varepsilon}$ defined on p. 637 coincide with the ones defined by the level sets of the function $\rho(x,v)=\sqrt{g(v,v)},$. Nevertheless, we give a short proof in the interests of completeness.

\begin{theorem}
\label{thm:identify}Suppose $M$ is real-analytically embedded into a complex manifold $X$ as a totally real submanifold of maximal dimension. Then there exists $R>0$ such that (1) for all $x\in M,$ the geometric exponential map $\exp_{x}$ extends to a holomorphic map of a ball of radius $R$ in $T_{x}^{\mathbb{C}}M$ into $X,$ (2) the map $(x,v)\rightarrow\exp_{x}(iv)$ is a diffeomorphism of $T^{R}M$ into $X,$ and (3) the pullback of the complex structure on $X$ to $T^{R}M$ by this map is the adapted complex structure on $T^{R}M.$

It follows that if $f:M\rightarrow\mathbb{C}$ has a holomorphic extension $f_{\mathbb{C}}$ to a neighborhood of $M$ in $X,$ the corresponding holomorphic extension of $f$ to $T^{R}M$ with respect to the adapted complex structure is given by $(x,v)\rightarrow f_{\mathbb{C}}(\exp_{x}(iv)).$
\end{theorem}

\begin{remark}
\label{ident.rem}If we take $X$ to be $T^{R}M$ itself with the adapted complex structure, then the identification in the theorem becomes simply an identity: $\exp_{x}(iv)=(x,v).$
\end{remark}

\begin{proof}
[Proof of Theorem \ref{thm:taylor}]It follows from the notion of adaptedness that the map $\sigma\longmapsto\exp_{x}(\sigma v)=\pi(\Phi_{\sigma}(x,v))$ has an analytic continuation in $\sigma$ to a strip for fixed $(x,v).$ The continuation is given by $\sigma+i\tau\longmapsto N_{\tau}\Phi_{\sigma}(x,v).$ Thus $t\mapsto\left(  f\circ\pi\circ\Phi_{t}\right)  (z)$ has a continuation given by $\sigma+i\tau\longmapsto f_{\mathbb{C}}(N_{\tau}\Phi_{\sigma}(x,v)).$

Now, the restriction of $E$ to each fiber is a homogeneous polynomial of degree 2, whereas the restriction to each fiber of $f\circ\pi$ is constant. It then follows by a simple inductive computation that the restriction to each fiber of $X_{E}^{k}(f\circ\pi)$ is a homogeneous polynomial of degree $k.$ When restricted to a line through the origin in $T_{x}M,$ (\ref{eqn:Taylor series}) is just the Taylor series of the map $t\mapsto\left(f\circ\pi\circ\Phi_{t}\right)  (z),$ which converges to the function in any disk lying within the strip.
\end{proof}

\medskip

\begin{proof}
[Proof of Theorem \ref{thm:identify}]
There exists a biholomorphism from some $T^{R}M$ with the adapted complex structure to a neighborhood $U$ of $M$ in $X.$ If we use this biholomorphism to identify $U$ with $T^{R}M,$ then it suffices to verify the identity in Remark \ref{ident.rem}. But if we analytically continue $\exp_{x},$ then the map $c\rightarrow\exp_{x}(cv)$ will be holomorphic for $c$ belonging to a neighborhood of the origin in $\mathbb{C}.$ So it suffices to analytically continue the map $\sigma\rightarrow\exp_{x}(\sigma v).$

Now, since the complex structure on $T^{R}M$ is ``adapted'' in the sense of Lempert--Sz\H{o}ke (Theorem \ref{adapted.thm}), the map sending $\sigma+i\tau$ to $(\gamma(\sigma ),\tau\dot{\gamma}(\sigma))$ is a holomorphic map of a strip in $\mathbb{C}$ into $T^{R}M.$ If $\gamma$ is the geodesic with initial value $x$ and initial derivative $v,$ then $\exp_{x}(\sigma v)$ is nothing but $\gamma(\sigma v),$ which we identify with $(\gamma(\sigma v),0)\in TM.$ Thus, the analytic continuation of the map $\sigma\rightarrow\exp_{x}(\sigma v)$ is the map $\sigma+i\tau\rightarrow(\gamma(\sigma),\tau\dot{\gamma}(\sigma)).$ Evaluating at $i$ gives $\exp_{x}(iv)=(\gamma(0),\dot{\gamma}(0))=(x,v).$
\end{proof}

\section*{Acknowledgments}
The authors are grateful to the referee for a very careful reading of the paper and numerous helpful observations, especially regarding the proofs of Theorems \ref{thm:continuation}, \ref{thm:integrable}, and \ref{thm:kaehler_potential}, as well as for pointing out sign errors in an earlier version, for several comments regarding the exposition, and for bringing to their attention the article \cite{Szoke95}.

{\small
\providecommand{\bysame}{\leavevmode\hbox to3em{\hrulefill}\thinspace}
\providecommand{\MR}{\relax\ifhmode\unskip\space\fi MR }
\providecommand{\MRhref}[2]{%
  \href{http://www.ams.org/mathscinet-getitem?mr=#1}{#2}
}
\providecommand{\href}[2]{#2}

}


\begin{thebibliography}{FMMN06}

\bibitem[ADW91]{Axelrod-DellaPietra-Witten}
S. Axelrod, S. {Della Pietra}, and E. Witten, \emph{Geometric
  quantization of {Chern-Simons} gauge theory}, J. Diff. Geom. \textbf{33}
  (1991), 787--902.

\bibitem[Agu01]{Aguilar01}
R.~M. Aguilar, \emph{Symplectic reduction and the homogeneous complex
  {M}onge-{A}mp\`ere equation}, Ann. Global Anal. Geom. \textbf{19} (2001),
  no.~4, 327--353.

\bibitem[BH01]{Burns-Hind01}
D. Burns and R. Hind, \emph{Symplectic geometry and the uniqueness of {G}rauert
  tubes}, Geom. Funct. Anal. \textbf{11} (2001), no.~1, 1--10.

\bibitem[FMMN05]{Florentino-Matias-Mourao-Nunes05}
C.~Florentino, P.~Matias, J.~Mour$\tilde{\text{a}}$o, and J.~P. Nunes,
  \emph{{Geometric quantization, complex structures and the coherent state
  transform}}, J. Func. Anal. \textbf{221} (2005), 303--322.

\bibitem[FMMN06]{Florentino-Matias-Mourao-Nunes06}
\bysame, \emph{{On the BKS Pairing for K\"ahler Quantizations for the Cotangent
  Bundle of a Lie Group}}, J. Func. Anal. \textbf{234} (2006), 180--198.

\bibitem[Gra58]{Grauert58}
H. Grauert, \emph{On {L}evi's problem and the imbedding of real-analytic
  manifolds}, Ann. of Math. (2) \textbf{68} (1958), 460--472.

\bibitem[GS91]{Guillemin-Stenzel1}
V. Guillemin and M. Stenzel, \emph{Grauert tubes and the homogeneous
  {M}onge-{A}mp\`ere equation}, J. Differential Geom. \textbf{34} (1991),
  no.~2, 561--570.

\bibitem[GS92]{Guillemin-Stenzel2}
\bysame, \emph{Grauert tubes and the homogeneous {M}onge-{A}mp\`ere equation.
  {II}}, J. Differential Geom. \textbf{35} (1992), no.~3, 627--641.

\bibitem[Hal94]{Hall94}
B.~C. Hall, \emph{{The {Segal--Bargmann} ``coherent state'' transform for
  compact {Lie} groups}}, J. Func. Anal. \textbf{122} (1994), 103--151.

\bibitem[Hal97]{Hall97}
\bysame, \emph{The inverse {S}egal-{B}argmann transform for compact {L}ie
  groups}, J. Funct. Anal. \textbf{143} (1997), no.~1, 98--116.

\bibitem[Hal02a]{Hall02}
\bysame, \emph{Geometric quantization and the generalized
  {Segal--Bargmann} transform for {Lie} groups of compact type}, Comm. Math.
  Phys. \textbf{226} (2002), 233--268.

\bibitem[Hal02b]{Halverscheid01}
S. Halverscheid, \emph{Complexifications of geodesic flows and adapted
  complex structures}, Rep. Math. Phys. \textbf{50} (2002), no.~3, 329--338.

\bibitem[HM02]{Hall-Mitchell02}
B.~C. Hall and J.~J. Mitchell, \emph{Coherent states on spheres}, J.
  Math. Phys. \textbf{43} (2002), no.~3, 1211--1236.

\bibitem[Hue06]{Huebschmann}
J. Huebschmann, \emph{K\"ahler quantization and reduction}, J reine
  angew. Math. \textbf{591} (2006), 75--109.

\bibitem[Jos95]{Jost}
J. Jost, \emph{Riemannian geometry and geometric analysis},
  Universitext, Springer-Verlag, Berlin, 1995.

\bibitem[Koh99]{Kohno}
M. Kohno, \emph{Global analysis in linear differential equations},
  Mathematics and its Applications, vol. 471, Kluwer Academic Publishers,
  Dordrecht, 1999.

\bibitem[LS91]{Lempert-Szoke}
L. Lempert and R. Sz{\H{o}}ke, \emph{Global solutions of
  the homogeneous complex {M}onge-{A}mp\`ere equation and complex structures on
  the tangent bundle of {R}iemannian manifolds}, Math. Ann. \textbf{290}
  (1991), no.~4, 689--712.

\bibitem[Sz{\H{o}}91]{Szoke}
R. Sz{\H{o}}ke, \emph{Complex structures on tangent bundles of
  {R}iemannian manifolds}, Math. Ann. \textbf{291} (1991), no.~3, 409--428.

\bibitem[Sz{\H{o}}95]{Szoke95}
\bysame, \emph{Automorphisms of certain Stein manifolds}, Math. Z.
  \textbf{219} (1995), no.~3, 357--385.

\bibitem[Sz{\H{o}}01]{Szoke01}
\bysame, \emph{Involutive structures on the tangent bundle of symmetric
  spaces}, Math. Ann. \textbf{319} (2001), no.~2, 319--348.

\bibitem[Thi96]{Thiemann96}
T. Thiemann, \emph{Reality conditions inducing transforms for quantum gauge
  field theory and quantum gravity}, Classical and Quantum Gravity \textbf{13}
  (1996), 1383– 1403.

\bibitem[Thi01]{Thiemann1}
\bysame, \emph{Gauge field theory coherent states ({GCS}). {I}.
  {G}eneral properties}, Classical Quantum Gravity \textbf{18} (2001), no.~11,
  2025--2064.

\bibitem[Thi06]{Thiemann06}
\bysame, \emph{Complexifier coherent states for quantum general relativity}, Classical and Quantum Gravity \textbf{23} (2006), no.~6, 2063--2117.

\bibitem[Tot03]{Totaro03}
B. Totaro, \emph{Complexifications of nonnegatively curved manifolds}, J.
  Eur. Math. Soc. (JEMS) \textbf{5} (2003), no.~1, 69--94.

\bibitem[WB59]{Bruhat-Whitney}
H.~Whitney and F.~Bruhat, \emph{Quelques propri\'et\'es fondamentales des
  ensembles analytiques-r\'eels}, Comment. Math. Helv. \textbf{33} (1959),
  132--160.

\end{thebibliography}
\end{document}